\documentclass[a4paper]{amsart}

\usepackage{amsmath}
\usepackage{url}
\usepackage{amsthm}
\usepackage{amsfonts}
\usepackage{appendix}
\usepackage{mathrsfs}
\usepackage{color}
\usepackage{enumerate}
\usepackage{amssymb}
\usepackage{tikz} 
\usetikzlibrary{calc}

\newtheorem{theorem}{Theorem}

\newtheorem{lemma}[theorem]{Lemma}
\newtheorem{proposition}[theorem]{Proposition}
\newtheorem{corollary}[theorem]{Corollary}
\newtheorem{maintheorem}{Main Theorem}

\theoremstyle{definition}
\newtheorem*{definition}{Definition}

\theoremstyle{remark}

\numberwithin{equation}{section}

\numberwithin{equation}{section}

\makeatletter
\newcommand*\owedge{\mathpalette\@owedge\relax}
\newcommand*\@owedge[1]{%
  \mathbin{%
    \ooalign{%
      $#1\m@th\bigcirc$\cr
      \hidewidth$#1\m@th\wedge$\hidewidth\cr
    }%
  }%
}
\makeatother

\makeatother
\numberwithin{equation}{section}

\title{Positivity of curvature on manifolds with boundary}

\author{Tsz-Kiu Aaron Chow}
\address{Department of Mathematics, Columbia University, New York, NY}

\email{achow@math.columbia.edu}

\begin{document}

\maketitle

\begin{abstract}
Consider a compact manifold $M$ with smooth boundary $\partial M$. Suppose that $g$ and $\tilde{g}$ are two Riemannian metrics on $M$. We construct a family of metrics on $M$ which agrees with $g$ outside a neighborhood of $\partial M$ and agrees with $\tilde{g}$ in a neighborhood of $\partial M$. We prove that the family of metrics preserves various natural curvature conditions under suitable assumptions on the boundary data. Moreover, under suitable assumptions on the boundary data, we can deform a metric to one with totally geodesic boundary while preserving various natural curvature conditions.
\end{abstract}

\maketitle

\section{Introduction}
\bigskip
In this paper, we study the interaction between boundary geometry and internal geometry of Riemannian manifolds. The study of such an interaction dates back to the work by Gromov and Lawson \cite{GL}, who observed that a metric of positive scalar curvature can be constructed on the double of a compact manifold with boundary which equips with metrics of positive scalar curvature and mean convex boundary. In the work of Shi and Tam \cite{ST}, the boundary behavior of compact manifolds with nonnegative scalar curvature was studied. They proved that the integral of mean curvature of the boundary of a spin manifold cannot be greater than the integral of mean curvature of its isometrically embedded convex image in Euclidean space. Their work implies positivity of the Brown-York quasilocal mass. On the other hand, the study of boundary effects to internal curvatures also plays an important role in studying manifolds with rough data. Positive mass theorem on manifolds with corners was proved by Miao \cite{Miao}. In his work, a smoothing procedure for the metric was done by constructing an "intrinsic bending" of Riemannian manifolds along the boundary that keeps the scalar curvature non-decreasing. More work were done for positive mass theorem on manifolds with corner in the literature such as the result by McFeron and Szekelyhidi \cite{MS}. Another important result that involves the study of boundary effects to internal curvatures is the counter example constructed by Brendle, Marques and Neves \cite{BMN} to the Min-Oo's conjecture. More specifically, a perturbation argument was performed in \cite{BMN} to make the boundary totally geodesic while keeping the scalar curvature non-decreasing. More related results can be found in a recent paper by Gromov \cite{Gm}.

 Inspired by these beautiful results, it is interesting to get a further understanding to how deformation of boundary data affects curvature conditions on the manifold. The goal of this paper is to study the influence of boundary deformation on various curvature conditions on the manifold. Positivity of curvature operator and isotropic curvature are natural curvature conditions that can be imposed on manifolds in higher dimensions. We recall the relavant definitions here:
 \newpage
 
 \begin{definition}\
 We say that
 \begin{itemize}
    \item[(i)] $(M,g)$
    has positive curvature operator if
    $R(\varphi,\varphi)>0$
    for all nonzero two-vectors $\varphi\in \bigwedge^2T_pM$ and all $p\in M$.\\
    \item [(ii)]  $(M,g)$ is (weakly) PIC if $R(z,w,\bar{z}.\bar{w}) > 0\ (\geq 0)$
    for all linearly independent vectors $z,w\in T_pM\otimes \mathbb{C}$ such that $g(z,z)=g(w,w)=g(z,w)=0$.\\
    \item[(iii)]   $(M,g)$ is (weakly) PIC1 if $R(z,w,\bar{z}.\bar{w})> 0\ (\geq 0)$
    for all linearly independent vectors $z,w\in T_pM\otimes \mathbb{C}$ such that     $g(z,z)g(w,w)-g(z,w)^2=0$.\\
    \item[(iv)]   $(M,g)$ is (weakly) PIC2 if $R(z,w,\bar{z}.\bar{w})> 0\ (\geq 0)$
    for all linearly independent vectors $z,w\in T_pM\otimes \mathbb{C}$. That means $(M,g)$ has nonnegative complex sectional curvature.
\end{itemize}	
 \end{definition}

Positive isotropic curvature was first introduced in the work of Micallef and Moore \cite{MM}, and variants of this condition were used to prove the Differentiable Sphere Theorem \cite{BS}. For manifolds without boundary, the importance of the curvature conditions above is that they are preserved by the Ricci Flow \cite{Bre08}. Manifolds with metrics satisfying these curvature conditions have rigid topological restrictions. For an initial metric that is strictly PIC, it was proved by Brendle in \cite{Bre19} that the Ricci flow only form neck-pinched singularities for $n\geq 12$. For an initial metric that is strictly PIC1, it was shown by Brendle \cite{Bre08} that the Ricci flow converges to a metric of constant curvature after rescaling. Given the significance of these curvature conditions, it is natural to study how they behave in manifolds with boundary. We are interested in preserving these curvature conditions while deforming the boundary data. To that end, We state the main theorem of this paper:

\begin{maintheorem}
Suppose that $M$ is a compact manifold with smooth boundary, and $g$ is a smooth Riemannian metric on $M$. Suppose that $\tilde{g}$ is another smooth Riemannian metric on $M$ such that $g=\tilde{g}$ on $\partial M$. Then there exists a family of smooth Riemannian metrics $\{\hat{g}_{\lambda}\}_{\lambda>\lambda^*}$ on $M$ and a neighborhood $U$ of $\partial M$ such that
\begin{itemize}
	\item $\hat{g}_{\lambda}$ agrees with $g$ outside $U$.
	\item $\hat{g}_{\lambda}$ agrees with $\tilde{g}$ in a neighborhood of $\partial M$.
\end{itemize}
Moreover, $\hat{g}_{\lambda}$ satisfies:
\begin{itemize}
\item[(i)] $\lim_{\lambda\to\infty}\hat{g}_{\lambda}=g$ in $C^{\alpha}$ for any $\alpha\in (0,1)$.
\item[(ii)] Suppose that $(M,g)$ has a convex boundary and that $(M, \tilde{g})$ has a  weakly convex boundary such that $A_g > A_{\tilde{g}} \geq 0$, then
	\begin{itemize}
	\item $(M,g)$ and $(M,\tilde{g})$ have positive curvature operator $\implies$ $(M,\hat{g}_{\lambda})$ has positive curvature operator;
    \item $(M,g)$ and $(M,\tilde{g})$ are PIC1 $\implies$ $(M,\hat{g}_{\lambda})$ is PIC1;
    \item $(M,g)$ and $(M,\tilde{g})$ are PIC2 $\implies$ $(M,\hat{g}_{\lambda})$ is PIC2.
	\end{itemize}	
\item[(iii)] Suppose that $(M,g)$ has a two-convex boundary and that $(M, \tilde{g})$ has a  weakly two-convex boundary such that 
	$$ A_g(X,X)+A_g(Y,Y) > A_{\tilde{g}}(X,X) + A_{\tilde{g}}(Y,Y) \geq 0 $$
	for all orthonormal $X,Y\in T(\partial M)$, then
	\begin{itemize}
	\item $(M,g)$ and $(M,\tilde{g})$ are PIC $\implies$ $(M,\hat{g}_{\lambda})$ is PIC.
	\end{itemize}
\item[(iv)] Suppose that $(M,g)$ has a mean-convex boundary and that $(M, \tilde{g})$ has a  weakly mean-convex boundary such that $H_g \geq H_{\tilde{g}} >0$, then
	\begin{itemize}
    \item $(M,g)$ and $(M,\tilde{g})$ have positive scalar curvature $\implies$ $(M,\hat{g}_{\lambda})$ has positive scalar curvature.
	\end{itemize}
\end{itemize}	
\end{maintheorem}

The author was informed that a similar result to the Main Theorem 1 was proved by Schlichting in his Ph.D. thesis \cite{Sch} which concerns with preserving curvature positivity on the gluing of two Riemannian manifolds along isometric boundaries. Nevertheless, our result allows the boundary to be two-convex for preserving positive isotropic curvature. The proof in Schlichting's thesis is also different from ours.

In particular, it would be helpful to preserve curvature conditions of a metric while deforming it to one with totally geodesic boundary. This procedure would be useful in smoothing manifolds with rough data. It is also believed that the above natural curvature conditions are preserved by Ricci flow on manifolds with boundary with the help of such a deformation. Here we confirm that such a deformation always exists under natural assumptions on the boundary data:

\begin{maintheorem}
Suppose that $M$ is a compact manifold with smooth boundary, and $g$ is a smooth Riemannian metric on $M$. Then there is a family of smooth Riemannian metrics $\{\hat{g}_{\lambda}\}_{\lambda>\lambda^*}$ on $M$ so that $\hat{g}_{\lambda} = g$ outside a neighborhood of $\partial M$ and satisfies:
\begin{itemize}
\item[(i)] $\lim_{\lambda\to\infty}\hat{g}_{\lambda}=g$ in $C^{\alpha}$ for any $\alpha\in [0,1)$.
\item[(ii)] $(M, \hat{g}_{\lambda})$ has a totally geodesic boundary.
\item[(iii)] If $(M,g)$ has a convex boundary, then
	\begin{itemize}
	\item $(M,g)$ has positive curvature operator $\implies$ $(M,\hat{g}_{\lambda})$ has positive curvature operator;
    \item $(M,g)$ is PIC1 $\implies$ $(M,\hat{g}_{\lambda})$ is PIC1;
    \item $(M,g)$ is PIC2 $\implies$ $(M,\hat{g}_{\lambda})$ is PIC2. 
	\end{itemize}	
\item[(iv)] If $(M,g)$ has a two-convex boundary, then
	\begin{itemize}
	\item $(M,g)$ is PIC $\implies$ $(M,\hat{g}_{\lambda})$ is PIC.
	\end{itemize}
\item[(v)] If $(M,g)$ has a mean-convex boundary, then
	\begin{itemize}
    \item $(M,g)$ has positive scalar curvature $\implies$ $(M,\hat{g}_{\lambda})$ has positve scalar curvature.
	\end{itemize}
\end{itemize}	
\end{maintheorem}

The proof of our results is based on choosing a suitable perturbation, and the family of metrics $\{\hat{g}_{\lambda}\}_{\lambda>\lambda^*}$ is defined by (2) in Section 3. Our work extends the result in \cite{BMN}.

\vskip 0.5cm
\noindent\textbf{Acknowledgement:}
The author would like to express his gratitude to his advisor Professor Simon Brendle for his continuing support, his guidance and many inspiring discussions.

\bigskip
\section{Auxiliary Results}
\bigskip

\begin{lemma}
Consider the Riemannian metrics of the form $\hat{g}=g+h$, where $h$ satisfies the pointwise estimate $|h|_g\leq\frac{1}{2}$. Let $p\in M$ be any point and $\{E_i\}$ be a (local) geodesic orthonormal frame with respect to $g$ around $p$, then the Riemann curvature tensor of $\hat{g}$ satisfies
\begin{align*}
    \hat{R}_{ijkl}=R_{ijkl}+ \frac{1}{2} \, g^{pq} \, R_{ijkp} \, h_{ql} - \frac{1}{2} \, g^{pq} \, R_{ijlp} \, h_{kq}+ E_{ijkl}+F_{ijkl}
\end{align*}
where
\begin{align*}
    E_{ijkl}=\frac{1}{2} \, \big [ -(D_{i,k}^2 h)_{jl} + (D_{i,l}^2 h)_{jl} + (D_{j,k}^2 h)_{il} - (D_{j,l}^2 h)_{ik} \big ]
\end{align*}
\begin{align*}
    F_{ijkl}=&\frac{1}{4} \, \hat{g}^{pq} \, [(D_j h)_{kp} + (D_k h)_{jp} - (D_p h)_{jk}] \, [(D_i h)_{lq} + (D_l h)_{iq} - (D_q h)_{il}] \\ 
&- \frac{1}{4} \, \hat{g}^{pq} \, [(D_i h)_{kp} + (D_k h)_{ip} - (D_p h)_{ik}] \, [(D_j h)_{lq} + (D_l h)_{jq} - (D_q h)_{jl}] 
\end{align*}
Here $D$ is the Levi-Civita connection with respect to $g$.
\end{lemma}
\bigskip

\begin{proof}
Let $g$ and $\hat{g}$ be two Riemannian metrics. Let $D$ denote the Levi-Civita connection associated with $g$, and let $\hat{D}$ denote the Levi-Civita connection associated with $\hat{g}$. Then 
\[\hat{D}_X Y = D_X Y + \Lambda(X,Y),\] 
where the tensor $\Lambda$ is given by 
\[2 \, \hat{g}(\Lambda(X,Y),Z) = (D_X \hat{g})(Y,Z) + (D_Y \hat{g})(X,Z) - (D_Z \hat{g})(X,Y).\]  
This gives 
\[\hat{D}_X \hat{D}_Y Z = D_X D_Y Z + \Lambda(X,D_Y Z) + D_X(\Lambda(Y,Z)) + \Lambda(X,\Lambda(Y,Z)).\] 
Consequently, 
\begin{align*} 
\hat{g}(\hat{D}_X \hat{D}_Y Z,W) 
&= \hat{g}(D_X D_Y Z,W) + \hat{g}(\Lambda(X,D_Y Z),W) \\ 
&\quad+ \hat{g}(D_X(\Lambda(Y,Z)),W) + \hat{g}(\Lambda(X,\Lambda(Y,Z)),W) \\ 
&= \hat{g}(D_X D_Y Z,W) + \hat{g}(\Lambda(X,D_Y Z),W) \\ 
&\quad+ X(\hat{g}(\Lambda(Y,Z),W)) - (D_X \hat{g})(\Lambda(Y,Z),W) \\ 
&\quad- \hat{g}(\Lambda(Y,Z),D_X W) + \hat{g}(\Lambda(X,\Lambda(Y,Z)),W).
\end{align*}
Using the identity $(D_X \hat{g})(U,V) = \hat{g}(\Lambda(X,U),V) + \hat{g}(U,\Lambda(X,V))$, we obtain 
\[(D_X \hat{g})(\Lambda(Y,Z),W) = \hat{g}(\Lambda(X,\Lambda(Y,Z)),W) + \hat{g}(\Lambda(Y,Z),\Lambda(X,W)),\] 
hence 
\begin{align*} 
\hat{g}(\hat{D}_X \hat{D}_Y Z,W) 
&= \hat{g}(D_X D_Y Z,W) + \hat{g}(\Lambda(X,D_Y Z),W) \\ 
&\quad + X(\hat{g}(\Lambda(Y,Z),W)) - \hat{g}(\Lambda(Y,Z),\Lambda(X,W)) \\ 
&\quad- \hat{g}(\Lambda(Y,Z),D_X W).
\end{align*}
In the following, we fix a point $p$, and work in geodesic normal coordinates around the point $p$. Putting $X = \partial_i$, $Y = \partial_j$, $Z = \partial_k$, $W = \partial_l$ gives 
\[\hat{g}(\hat{D}_{\partial_i} \hat{D}_{\partial_j} \partial_k,\partial_l) = \hat{g}(D_{\partial_i} D_{\partial_j} \partial_k,\partial_l) + \partial_i(\hat{g}(\Lambda(\partial_j,\partial_k),\partial_l)) - \hat{g}(\Lambda(\partial_j,\partial_k),\Lambda(\partial_i,\partial_l))\] 
at $p$. Next, we switch $i$ and $j$ and take the difference. Note that 
\[-\hat{g}(\hat{D}_{\partial_i} \hat{D}_{\partial_j} \partial_k,\partial_l) + \hat{g}(\hat{D}_{\partial_j} \hat{D}_{\partial_i} \partial_k,\partial_l) = \hat{R}_{ijkl}\] 
and 
\[-\hat{g}(D_{\partial_i} D_{\partial_j} \partial_k,\partial_l) + \hat{g}(D_{\partial_j} D_{\partial_i} \partial_k,\partial_l) = g^{pq} \, R_{ijkp} \, \hat{g}_{ql}\] 
at $p$. Moreover, using the definition of $\Lambda$, we obtain 
\begin{align*} 
&- \partial_i(\hat{g}(\Lambda(\partial_j,\partial_k),\partial_l)) + \partial_j(\hat{g}(\Lambda(\partial_i,\partial_k),\partial_l)) \\ 
&= -\frac{1}{2} \, \big [ (D_{i,j}^2 \hat{g})_{kl} + (D_{i,k}^2 \hat{g})_{jl} - (D_{i,l}^2 \hat{g})_{jk} \big ] + \frac{1}{2} \, \big [ (D_{j,i}^2 \hat{g})_{kl} + (D_{j,k}^2 \hat{g})_{il} - (D_{j,l}^2 \hat{g})_{ik} \big ] \\ 
&= -\frac{1}{2} \, g^{pq} \, R_{ijkp} \, \hat{g}_{ql} - \frac{1}{2} \, g^{pq} \, R_{ijlp} \, \hat{g}_{kq} \\ 
&\quad+ \frac{1}{2} \, \big [ -(D_{i,k}^2 \hat{g})_{jl} + (D_{i,l}^2 \hat{g})_{jl} + (D_{j,k}^2 \hat{g})_{il} - (D_{j,l}^2 \hat{g})_{ik} \big ] \\ 
\end{align*} 
at $p$. Finally, 
\begin{align*}
&\hat{g}(\Lambda(\partial_j,\partial_k),\Lambda(\partial_i,\partial_l)) - \hat{g}(\Lambda(\partial_i,\partial_k),\Lambda(\partial_j,\partial_l)) \\ 
&= \frac{1}{4} \, \hat{g}^{pq} \, [(D_j \hat{g})_{kp} + (D_k \hat{g})_{jp} - (D_p \hat{g})_{jk}] \, [(D_i \hat{g})_{lq} + (D_l \hat{g})_{iq} - (D_q \hat{g})_{il}] \\ 
&\quad- \frac{1}{4} \, \hat{g}^{pq} \, [(D_i \hat{g})_{kp} + (D_k \hat{g})_{ip} - (D_p \hat{g})_{ik}] \, [(D_j \hat{g})_{lq} + (D_l \hat{g})_{jq} - (D_q \hat{g})_{jl}] 
\end{align*} 
at $p$. Putting these facts together, we obtain 
\begin{align*} 
\hat{R}_{ijkl} 
&= \frac{1}{2} \, g^{pq} \, R_{ijkp} \, \hat{g}_{ql} - \frac{1}{2} \, g^{pq} \, R_{ijlp} \, \hat{g}_{kq} \\ 
&\quad+ \frac{1}{2} \, \big [ -(D_{i,k}^2 \hat{g})_{jl} + (D_{i,l}^2 \hat{g})_{jl} + (D_{j,k}^2 \hat{g})_{il} - (D_{j,l}^2 \hat{g})_{ik} \big ] \\ 
&\quad+ \frac{1}{4} \, \hat{g}^{pq} \, [(D_j \hat{g})_{kp} + (D_k \hat{g})_{jp} - (D_p \hat{g})_{jk}] \, [(D_i \hat{g})_{lq} + (D_l \hat{g})_{iq} - (D_q \hat{g})_{il}] \\ 
&\quad- \frac{1}{4} \, \hat{g}^{pq} \, [(D_i \hat{g})_{kp} + (D_k \hat{g})_{ip} - (D_p \hat{g})_{ik}] \, [(D_j \hat{g})_{lq} + (D_l \hat{g})_{jq} - (D_q \hat{g})_{jl}] 
\end{align*} 
at $p$. Hence, if we put $h = \hat{g}-g$, then 
\begin{align*} 
\hat{R}_{ijkl} 
&= R_{ijkl} + \frac{1}{2} \, g^{pq} \, R_{ijkp} \, h_{ql} - \frac{1}{2} \, g^{pq} \, R_{ijlp} \, h_{kq} \\ 
&\quad+ \frac{1}{2} \, \big [ -(D_{i,k}^2 h)_{jl} + (D_{i,l}^2 h)_{jl} + (D_{j,k}^2 h)_{il} - (D_{j,l}^2 h)_{ik} \big ] \\ 
&\quad+ \frac{1}{4} \, \hat{g}^{pq} \, [(D_j h)_{kp} + (D_k h)_{jp} - (D_p h)_{jk}] \, [(D_i h)_{lq} + (D_l h)_{iq} - (D_q h)_{il}] \\ 
&\quad- \frac{1}{4} \, \hat{g}^{pq} \, [(D_i h)_{kp} + (D_k h)_{ip} - (D_p h)_{ik}] \, [(D_j h)_{lq} + (D_l h)_{jq} - (D_q h)_{jl}] 
\end{align*} 
at $p$.

\end{proof}

\bigskip

\begin{lemma} Let $S$ and $T$ be real symmetric bilinear forms on $T_xM$ . If $S$ and $T$ are (weakly) positive, then $S\owedge T$ is (weakly) positive on $\bigwedge^2T_xM^{\mathbb{C}}$.
		
\end{lemma}

\begin{proof}
By assumption, we can write $S, T$ as a sum of rank-one operators:
\begin{align*}
	S = \sum_{a=1}^n v^a\otimes v^a, \quad T = \sum_{b=1}^n w^b\otimes w^b.	
\end{align*}
 Then for any two-vectors $\varphi = \varphi^{ij}e_i\wedge e_j\in\bigwedge^2T_xM^{\mathbb{C}}$, we have
		\begin{align*}
			(S\owedge T) (\varphi, \bar{\varphi})\ &=\ 4\varphi^{ij}\bar{\varphi}^{kl}S_{ik}T_{jl}\\
			&=\ 4\sum_{a,b}\varphi^{ij}\bar{\varphi}^{kl}v^a_iv^a_kw^b_jw^b_l\\
			&=\ 4 \sum_{a,b}|\varphi^{ij}v^a_iw^b_j|^2\\
			&\geq\ 0.
		\end{align*}
\end{proof}
\bigskip

\begin{lemma}[\text{c.f. \cite{Bre19}, Lemma A.2} ] Let $S$ and $T$ be real symmetric bilinear forms on $T_xM$. If $S$ and $T$ are (weakly) two-positive with respect to the metric $g$, then $S\owedge T$ is (weakly) {\rm PIC} with respect to the metric $g$.
		
\end{lemma}

\begin{proof}
	Let $\zeta, \eta\in T_xM^{\mathbb{C}}$ be linear independent vectors such that
		$$ g(\zeta, \zeta) = g(\zeta, \eta) = g(\eta, \eta) = 0. $$
	We shall show that $(S\owedge T) (\zeta, \eta, \bar{\zeta}, \bar{\eta}) \geq 0$. We can find vectors $z, w\in {\rm span}\{\zeta,\eta\}$ such that $g(z, \bar{z}) = g(w, \bar{w}) = 1$, $g(z, \bar{w}) = 0$ and $S(z, \bar{w}) = 0$. The identities $ g(\zeta, \zeta) = g(\zeta, \eta) = g(\eta, \eta) = 0$ give $ g(z, z) = g(z, w) = g(w, w) = 0$. Consequently, we can find an orthonormal four-frame $\{e_1, e_2, e_3, e_4\}\subset T_xM$ such that
		$$ z = e_1 + i e_2,\quad w = e_3 + ie_4. $$
		Using the identities $S(\bar{z}, w) = S(z, \bar{w}) = 0$, we obtain
		\begin{align*}
			(S\owedge T) (z, w, \bar{z}, \bar{w})\ &=\ S(z, \bar{z}) T(w, \bar{w})\ +\ S(w, \bar{w}) T(z, \bar{z})\\
			&=\ (S_{11} + S_{22})(T_{33} + T_{44})\ +\ (S_{33} + S_{44})(T_{11} + T_{22})\\
			&\geq\ 0.
		\end{align*}
	Since ${\rm span}\{z, w\} = {\rm span}\{\zeta,\eta\}$, we conclude that $(S\owedge T) (\zeta, \eta, \bar{\zeta}, \bar{\eta}) \geq 0$.
\end{proof}
\bigskip

\bigskip

\section{Preserving Curvature Conditions}

\bigskip

In this section, $g$ and $\tilde{g}$ are Riemannian metrics on $M$ such that $g - \tilde{g} = 0$ along $\partial M$. Now, we describe our choice of perturbation as in \cite{BMN}. We fix a neighborhood $U$ of $\partial M$ and a smooth boundary defining function $\rho:M\to [0,\infty)$ by taking it to be the distance function from $\partial M$ with respect to the metric $g$. Then we have $|D\rho|_g =1$. Since $g-\tilde{g}=0$ along $\partial M$, we can find a symmetric (0,2)-tensor $S$ such that  $\tilde{g}=g+\rho S$ in a neighborhood of $\partial M$ and $S=0$ outside $U$. Then for all $X,Y\in T(\partial M)$, the second fundamental forms satisfy
\begin{align*}
    &\frac{1}{2}S(X,Y)=A_g(X,Y)-A_{\tilde{g}}(X,Y),\\
    &D^2\rho(X,Y) = -A_g(X,Y).\notag
\end{align*}
This implies that the identity
\begin{align}
    &A_g(X,Y) - A_{\tilde{g}}(X,Y)=\frac{1}{2}S(X,Y)=-D^2\rho(X,Y)-A_{\tilde{g}}(X,Y)
\end{align}
holds on the boundary $\partial M$ for all $X,Y\in T(\partial M)$.

\bigskip

Let $\chi:[0,\infty)\to[0,1]$ a smooth cut-off function with the following properties (c.f. \cite{BMN}, Lemma 17):
\begin{itemize}
    \item $\chi(s)=s-\frac{1}{2}s^2$ for $s\in[0,\frac{1}{2}]$;
    \item $\chi(s)$ is constant for $s\geq 1$;
    \item $\chi''(s)<0$ for $s\in[0,1)$.
\end{itemize}
\bigskip
Moreover, let $\beta:(-\infty,0]\to[0,1]$ be a smooth cutoff function such that

\begin{itemize}
    \item $\beta(s)=\frac{1}{2}$ for $s\in [-1,0]$;
    \item $\beta(s)=0$ for $s\in (-\infty,-2]$.
\end{itemize}
\bigskip 
Now, if $\lambda>0$ is sufficiently large, we can define a smooth metric $\hat{g}_{\lambda}$ on $M$ by 
\begin{align}
    \hat{g}_{\lambda}=
    \begin{cases}
    g+\lambda^{-1}\chi(\lambda\rho)S\quad\quad\quad\quad &\text{for}\quad \rho\geq e^{-\lambda^2}\\
    \tilde{g}-\lambda\rho^2\beta(\lambda^{-2}\log\rho)S\quad &\text{for}\quad\rho<e^{-\lambda^2}.
    \end{cases}
\end{align}
In the sequel, we will show that $\hat{g}_{\lambda}$ preserves various curvature conditions of $g$ and $\tilde{g}$ for sufficiently large $\lambda$. Note that we have $\hat{g}_{\lambda}=\tilde{g}$ in the region $\{\rho\leq e^{-2\lambda^2}\}$ and $\hat{g}_{\lambda}=g$ outside $U$. 

\bigskip
Now we give a lower bound to the curvature operator of $\hat{g}_{\lambda}$. We first consider the region $\{\rho\geq e^{-\lambda^2}\}$.

\bigskip
\begin{proposition}
 Suppose that $(M,g)$ has a convex boundary and that $(M, \tilde{g})$ has a  weakly convex boundary such that $A_g > A_{\tilde{g}} \geq 0$. Let $\epsilon$ be an arbitrary positive real number. If $\lambda>0$ is sufficiently large, then 
\begin{align*}
    Rm_{\hat{g}_{\lambda}}(x)(\varphi,\bar{\varphi})-Rm_{g}(x)(\varphi,\bar{\varphi})\geq -\epsilon|\varphi|^2_g
\end{align*}
for any (2,0)-tensor $\varphi\in\bigwedge^2T_xM^{\mathbb{C}}$ and any $x\in M$ in the region $\{\rho(x)\geq e^{-\lambda^2}\}$. 
\end{proposition}
\bigskip
\begin{proof}
First we fixed a point $x\in M$ so that $\rho(x)\geq e^{-\lambda^2}$. Let $\{e_i\}$ be a geodesic normal frame around $x$. Without loss of generality let us assume that the two-vector  $\varphi\in\bigwedge^2T_xM^{\mathbb{C}}$ is normalized so that $|\varphi|_g=1$. We can write $\varphi=\sum_{i,j}\varphi^{ij}e_i\wedge e_j$ so that $\varphi^{ij}$ is anti-symmetric. Einstein summation will be adopted freely so that $\varphi^{ij}A_{jk}$ means that we are summing over $j=1,..,n$..  In the region $\{\rho(x)\geq e^{-\lambda^2}\}$, we have $\hat{g}_{\lambda}=g+h_{\lambda}$, where
\begin{align*}
    h_{\lambda}=\lambda^{-1}\chi(\lambda\rho)S.
\end{align*}
The tensor $h_{\lambda}$ satisfies
\begin{align}
    (D_{e_j}h_{\lambda})(e_k,e_l)&=\chi'(\lambda\rho)D_j\rho\cdot S_{kl} + \lambda^{-1}\chi(\lambda\rho)D_jS_{kl}
\end{align}
and 
\begin{align} 
    (D^2_{e_i,e_j}h_{\lambda})(e_k,e_l)&= \lambda\chi''(\lambda\rho)D_i\rho D_j\rho\cdot S_{kl} + \chi'(\lambda\rho)D_iD_j\rho\cdot S_{kl}\\
    &\quad +\chi'(\lambda\rho)D_j\rho\cdot D_iS_{kl} + \chi'(\lambda\rho)D_i\rho\cdot D_jS_{kl}\notag\\
    &\quad +\lambda^{-1}\chi(\lambda\rho)D_iD_jS_{kl}\notag.
\end{align}

\bigskip
Since $\varphi$ is a (2,0)-tensor, $\varphi$ induces a linear map $[\varphi]:T_xM^*\to T_xM$ via the action $[\varphi]w := \varphi^{ij}w(e_i)e_j$. Using the notation of Lemma 2.1, we compute

\begin{align*} 
\varphi^{ij} \bar{\varphi}^{kl} E_{ijkl} 
&= -2 \varphi^{ij} \bar{\varphi}^{kl} (D_{i,k}^2 h)_{jl} \\ 
&= -2 \lambda \chi''(\lambda\rho) S([\varphi] d\rho,\overline{[\varphi] d\rho}) - 2 \varphi^{ij} \bar{\varphi}^{kl} \chi'(\lambda\rho) D_i D_k \rho S_{jl} \\ 
&\quad - 2 \varphi^{ij} \bar{\varphi}^{kl} \chi'(\lambda\rho) (D_k \rho D_i S_{jl} + D_i \rho D_k S_{jl}) + O(\lambda^{-1}) 
\end{align*} 
and 
\begin{align*} 
&\varphi^{ij} \bar{\varphi}^{kl} F_{ijkl} \\
&= \frac{1}{2} \varphi^{ij} \bar{\varphi}^{kl} \hat{g}^{pq} (D_j h_{kp}+D_k h_{jp} - D_p h_{jk}) \\ 
&\hspace{25mm} \cdot (D_i h_{lq} + D_l h_{iq} - D_q h_{il}) \\ 
&= \frac{1}{2} \varphi^{ij} \bar{\varphi}^{kl} \chi'(\lambda\rho)^2 \hat{g}^{pq} (D_j \rho S_{kp}+D_k \rho S_{jp} - D_p \rho S_{jk}) \\ 
&\hspace{35mm} \cdot (D_i \rho S_{lq} + D_l \rho S_{iq} - D_q \rho S_{il}) + O(\lambda^{-1}) \\ 
&=  \varphi^{ij} \bar{\varphi}^{kl} \chi'(\lambda\rho)^2 \hat{g}^{pq} \Big [ D_j \rho D_l \rho S_{kp} S_{iq} - \frac{1}{2} D_p \rho D_q \rho S_{ik} S_{jl} \\ 
&\hspace{35mm} + (D_j \rho S_{lp} + D_l \rho S_{jp}) D_q \rho S_{ik} \Big ] + O(\lambda^{-1}).
\end{align*}
Using Lemma 2.1, we obtain 
\begin{align*} 
&\varphi^{ij} \bar{\varphi}^{kl} (R_{\hat{g}_\lambda})_{ijkl} - \varphi^{ij} \bar{\varphi}^{kl} R_{ijkl} \\ 
&= \varphi^{ij} \bar{\varphi}^{kl} E_{ijkl} + \varphi^{ij} \bar{\varphi}^{kl} F_{ijkl} + O(\lambda^{-1}) \\
&= 2 \lambda (-\chi''(\lambda\rho)) S([\varphi] d\rho,\overline{[\varphi] d\rho}) \\ 
&\quad+ \varphi^{ij} \bar{\varphi}^{kl} \chi'(\lambda\rho) (-2 D_i D_k \rho \, S_{jl} - \frac{1}{2} \chi'(\lambda\rho) |D\rho|_{\hat{g}_\lambda}^2 S_{ik} S_{jl}) \\ 
&\quad- 2 \varphi^{ij} \bar{\varphi}^{kl} \chi'(\lambda\rho) (D_i \rho \, D_k S_{jl} + D_k \rho \, D_i S_{jl}) \\ 
&\quad+ \varphi^{ij} \bar{\varphi}^{kl} \chi'(\lambda\rho)^2 \hat{g}^{pq} (D_j \rho S_{lp} + D_l \rho S_{jp}) D_q \rho S_{ik} \\ 
&\quad+ \varphi^{ij} \bar{\varphi}^{kl} \chi'(\lambda\rho)^2 \hat{g}^{pq} D_j \rho D_l \rho S_{kp} S_{iq} + O(\lambda^{-1}). 
\end{align*} 
This gives 
\begin{align} 
&\varphi^{ij} \bar{\varphi}^{kl} (R_{\hat{g}_\lambda})_{ijkl} - \varphi^{ij} \bar{\varphi}^{kl} R_{ijkl} \\ 
&\notag\geq 2 \lambda (-\chi''(\lambda\rho)) S([\varphi] d\rho,\overline{[\varphi] d\rho}) \\ 
&\notag\quad+ \varphi^{ij} \bar{\varphi}^{kl} \chi'(\lambda\rho) (-2 D_i D_k \rho \, S_{jl} - \frac{1}{2} \chi'(\lambda\rho) |D\rho|_{\hat{g}_\lambda}^2 S_{ik} S_{jl}) \\ 
&\notag\quad- C \chi'(\lambda\rho) |[\varphi] d\rho| - C \lambda^{-1}, 
\end{align} 
where $C$ is independent of $\lambda$. By assumption, $A_g - A_{\tilde{g}} > 0$ along the $\partial M$. Therefore, the restriction of $S$ to the tangent space of $\partial M$ is positive definite. Hence, we can find a tensor $\tilde{S}$ such that $\tilde{S}-S$ is a large (but fixed) multiple of $d\rho \otimes d\rho$, and $\tilde{S}$ is positive definite at each point on $\partial M$. Let us fix a small number $a>0$ such that $\tilde{S}-ag$ is positive definite in a small neighborhood of $\partial M$. This implies 
\[\tilde{S}([\varphi] d\rho,\overline{[\varphi] d\rho}) \geq a \, |[\varphi] d\rho|^2\] 
in a neighborhood of $\partial M$. On the other hand, since $[\varphi]$ is anti-symmetric, we have $([\varphi] d\rho)^i D_i \rho = \varphi^{ij} D_j \rho D_i \rho = 0$. In other words, the vector $[\varphi] d\rho$ is annihilated by the one-form $d\rho$. Since $\tilde{S}-S$ is a multiple of $d\rho \otimes d\rho$, it follows that 
\begin{align}
	S([\varphi] d\rho,\overline{[\varphi] d\rho}) = \tilde{S}([\varphi] d\rho,\overline{[\varphi] d\rho}) \geq a \, |[\varphi] d\rho|^2
\end{align}
in a neighborhood of $\partial M$. The above estimate with (5) give

\begin{align} 
&\varphi^{ij} \bar{\varphi}^{kl} (R_{\hat{g}_\lambda})_{ijkl} - \varphi^{ij} \bar{\varphi}^{kl} R_{ijkl} \\ 
&\notag\geq 2a \lambda (-\chi''(\lambda\rho)) \, |[\varphi] d\rho|^2 \\ 
&\notag\quad+ \varphi^{ij} \bar{\varphi}^{kl} \chi'(\lambda\rho) (-2 D_i D_k \rho \, S_{jl} - \frac{1}{2} \chi'(\lambda\rho) |D\rho|_{\hat{g}_\lambda}^2 S_{ik} S_{jl}) \\ 
&\notag\quad- C \chi'(\lambda\rho) |[\varphi] d\rho| - C \lambda^{-1}. 
\end{align} 
Since $\chi(0)=0$, we can find a real number $s_0\in[0,1)$ such that
\begin{align}
\chi'(s_0) \sup_U\Big|\varphi^{ij} \bar{\varphi}^{kl}  (2 D_i D_k \rho \, S_{jl} + \frac{1}{2} \chi'(\lambda\rho) |D\rho|_{\hat{g}_\lambda}^2 S_{ik} S_{jl}) + C |[\varphi] d\rho| \Big| < \epsilon.
\end{align}
In the region $\{\rho\geq s_0\lambda^{-1}\}$, we have
\begin{align*} 
&\varphi^{ij} \bar{\varphi}^{kl} (R_{\hat{g}_\lambda})_{ijkl} - \varphi^{ij} \bar{\varphi}^{kl} R_{ijkl} \\ 
&\notag\geq 2a \lambda (-\chi''(\lambda\rho)) \, |[\varphi] d\rho|^2 \\ 
&\notag\quad- \sup_{s\geq s_0}\chi'(s) \sup_U\Big|\varphi^{ij} \bar{\varphi}^{kl}  (2 D_i D_k \rho \, S_{jl} + \frac{1}{2} \chi'(\lambda\rho) |D\rho|_{\hat{g}_\lambda}^2 S_{ik} S_{jl}) + C |[\varphi] d\rho| \Big|  - C \lambda^{-1}\\
&\notag = 2a \lambda (-\chi''(\lambda\rho)) \, |[\varphi] d\rho|^2 \\ 
&\notag\quad- \chi'(s_0) \sup_U\Big|\varphi^{ij} \bar{\varphi}^{kl}  (2 D_i D_k \rho \, S_{jl} + \frac{1}{2} \chi'(\lambda\rho) |D\rho|_{\hat{g}_\lambda}^2 S_{ik} S_{jl}) + C |[\varphi] d\rho| \Big|  - C \lambda^{-1}. 
\end{align*} 
Thus it follows from (8) that 
\begin{align*}
    \inf_{\rho\geq s_0\lambda^{-1}}(Rm_{\hat{g}_{\lambda}}(\varphi,\bar{\varphi})-Rm_{g}(\varphi,\bar{\varphi}))\geq -\epsilon
\end{align*}
if $\lambda>0$ is sufficiently large. 
\bigskip

We next consider the region $\{e^{-\lambda^2} \leq \rho \leq s_0\lambda^{-1}\}$. We can find a neighborhood $V$ of $\partial M$ and a diffeomorphism $\Phi$ such that $\Phi:V\cong \partial M\times [0,\delta)$. Let $T$ be a smooth (0,2)-tensor in a neighborhood of $\partial M$ defined by $T = \Phi^*(A_{\tilde{g}} + d\rho\otimes d\rho)$. Thus $T$ is weakly positive definite in a neighborhood of $\partial M$.  We observe by (1) that the restriction of $S_{ik} + 2 D_i D_k \rho + T_{ik}$ to the tangent space to $\partial M$ vanishes. Therefore, we may write $S_{ik} + 2 D_i D_k \rho + T_{ik} = \omega_i D_k \rho + \omega_k D_i \rho$ at each point on $\partial M$, where $\omega$ is a suitable $1$-form. Hence, in a small neighborhood of the boundary, we have $|S_{ik} + 2 D_i D_k \rho + T_{ik} - \omega_i D_k \rho - \omega_k D_i \rho| \leq C \rho$. This implies 
\begin{align*} 
|\varphi^{ij} \bar{\varphi}^{kl} (S_{ik}+2 D_iD_k \rho + T_{ik}) S_{jl}| 
&\leq C |[\varphi] d\rho| + C \rho \\ 
&\leq C |[\varphi] d\rho| + C \lambda^{-1} 
\end{align*} 
in the region $\{e^{-\lambda^2} \leq \rho \leq s_0\lambda^{-1}\}$. Putting these facts with (7) together, we obtain 
\begin{align*} 
&\varphi^{ij} \bar{\varphi}^{kl} (R_{\hat{g}_\lambda})_{ijkl} - \varphi^{ij} \bar{\varphi}^{kl} R_{ijkl} \\ 
&\geq 2a \lambda (-\chi''(\lambda\rho)) \, |[\varphi] d\rho|^2 \\ 
&\quad + \varphi^{ij} \bar{\varphi}^{kl} \chi'(\lambda\rho) (1 - \frac{1}{2} \chi'(\lambda\rho) |D\rho|_{\hat{g}_\lambda}^2) S_{ik} S_{jl} \\ 
&\quad + \varphi^{ij} \bar{\varphi}^{kl} \chi'(\lambda\rho) T_{ik} S_{jl} - C \chi'(\lambda\rho) |[\varphi] d\rho| - C \lambda^{-1} 
\end{align*} 
in the region $\{e^{-\lambda^2} \leq \rho \leq s_0\lambda^{-1}\}$. 

Next, we may assume without loss of generality that the frame $\{e_i\}$ diagonalizes the (0,2)-tensor $\tilde{S}$. Since $\tilde{S}-ag$ is positive definite and $T$ is weakly positive definite, we have
\begin{align*}
	\varphi^{ij}\bar{\varphi}^{kl}T_{ik}\tilde{S}_{jl} \geq 0
\end{align*}
in a neighborhood of $\partial M$ by Lemma 2.2. Since $\tilde{S}-S$ is a multiple of $d\rho \otimes d\rho$, it follows that
\begin{align*}
		\varphi^{ij}\bar{\varphi}^{kl}T_{ik}S_{jl} \geq \varphi^{ij}\bar{\varphi}^{kl}T_{ik}\tilde{S}_{jl} -  C\, |[\varphi] d\rho| \geq  - C \, |[\varphi] d\rho|
\end{align*}
in a neighborhood of $\partial M$.

On the other hand, since $\tilde{S}-ag$ is positive definite, we know that 
\[\varphi^{ij} \bar{\varphi}^{kl} \tilde{S}_{ik} \tilde{S}_{jl} \geq a^2 |\varphi|^2 = a^2\] 
in a neighborhood of $\partial M$. Since $\tilde{S}-S$ is a multiple of $d\rho \otimes d\rho$, it follows that 
\begin{align*}
	\varphi^{ij} \bar{\varphi}^{kl} S_{ik} S_{jl} \geq \varphi^{ij} \bar{\varphi}^{kl} \tilde{S}_{ik} \tilde{S}_{jl}  - C \, |[\varphi] d\rho| \geq a^2 - C \, |[\varphi] d\rho|
\end{align*}
in a neighborhood of $\partial M$. To summarize, we obtain  
\begin{align} 
&\varphi^{ij} \bar{\varphi}^{kl} (R_{\hat{g}_\lambda})_{ijkl} - \varphi^{ij} \bar{\varphi}^{kl} R_{ijkl} \\ 
&\notag\geq 2a \lambda (-\chi''(\lambda\rho)) \, |[\varphi] d\rho|^2 + a^2 \chi'(\lambda\rho) (1 - \frac{1}{2} \chi'(\lambda\rho) |D\rho|_{\hat{g}_\lambda}^2) \\ 
&\notag\quad- N \chi'(\lambda\rho) |[\varphi] d\rho| - N \lambda^{-1}\\
&\notag = |[\varphi] d\rho| \Big[ 2a \lambda (-\chi''(\lambda\rho)) \, |[\varphi] d\rho| - N \chi'(\lambda\rho) \Big]\\
&\notag\quad + a^2 \chi'(\lambda\rho) (1 - \frac{1}{2} \chi'(\lambda\rho) |D\rho|_{\hat{g}_\lambda}^2) - N\lambda^{-1}
\end{align} 
in the region $\{e^{-\lambda^2} \leq \rho \leq s_0\lambda^{-1}\}$.  Here, $N$ is a positive constant depending only on $(M,g)$ but independent of $\lambda$. This implies that
\begin{align*}
    Rm_{\hat{g}_{\lambda}}(\varphi,\bar{\varphi})-Rm_{g}(\varphi,\bar{\varphi})\geq -\frac{N^2\chi'(\lambda\rho)^2}{8a\lambda\cdot\inf_{0\leq s\leq s_0}(-\chi''(s))}+ a^2 \chi'(\lambda\rho) (1 - \frac{1}{2} |D\rho|_{\hat{g}_\lambda}^2) -N\lambda^{-1}.
\end{align*}
Here we have used $x(Ax-B)\geq -\frac{B^2}{4A}$ for positive $A$ and the fact that $\chi$ is between $0$ and $1$. Since $|D\rho|_{\hat{g}_{\lambda}} = |D\rho|_g = 1$ on the boundary $\partial M$, by continuity we have $(1 - \frac{1}{2} |D\rho|_{\hat{g}_\lambda}^2)>\frac{1}{4}$ in the region $\{e^{-\lambda^2} \leq \rho \leq s_0\lambda^{-1}\}$ if $\lambda$ is sufficiently large. By the definition of $\chi$, we also have $\inf_{0\leq s\leq s_0}(-\chi''(s))>0$. Thus, we conclude that

\begin{align*}
    \inf_{e^{-\lambda^2}\leq \rho\leq s_0\lambda^{-1}}(Rm_{\hat{g}_{\lambda}}(\varphi,\bar{\varphi})-Rm_{g}(\varphi,\bar{\varphi}))\geq \frac{1}{4}a^2\chi'(\lambda\rho)
\end{align*}
if $\lambda$ is sufficiently large. Combing these facts, we complete the proof.
\end{proof}

\bigskip

\begin{corollary}
 Suppose that $(M,g)$ has a two-convex boundary and that $(M, \tilde{g})$ has a  weakly two-convex boundary such that 
	$$ A_g(X,X)+A_g(Y,Y) > A_{\tilde{g}}(X,X) + A_{\tilde{g}}(Y,Y) \geq 0 $$
	for all orthonormal $X,Y\in T(\partial M)$. Let $\epsilon$ be an arbitrary positive real number. If $\lambda>0$ is sufficiently large, then 
\begin{align*}
    Rm_{\hat{g}_{\lambda}}(x)(\varphi,\bar{\varphi})-Rm_{g}(x)(\varphi,\bar{\varphi})\geq -\epsilon|\varphi|_g^2
\end{align*}
for any isotropic 2-vector $\varphi=z\wedge w\in\bigwedge^2T_xM^{\mathbb{C}}$ satisfying $\hat{g}_{\lambda}(z,z)=\hat{g}_{\lambda}(w,w)=\hat{g}_{\lambda}(z,w)=0$ and any $x\in M$ in the region $\{\rho(x)\geq e^{-\lambda^2}\}$. 
\end{corollary}
\bigskip
\begin{proof}
As in Proposition 3.1, we recall (5):
\begin{align*} 
&\varphi^{ij} \bar{\varphi}^{kl} (R_{\hat{g}_\lambda})_{ijkl} - \varphi^{ij} \bar{\varphi}^{kl} R_{ijkl} \\ 
&\notag\geq 2 \lambda (-\chi''(\lambda\rho)) S([\varphi] d\rho,\overline{[\varphi] d\rho}) \\ 
&\notag\quad+ \varphi^{ij} \bar{\varphi}^{kl} \chi'(\lambda\rho) (-2 D_i D_k \rho \, S_{jl} - \frac{1}{2} \chi'(\lambda\rho) |D\rho|_{\hat{g}_\lambda}^2 S_{ik} S_{jl}) \\ 
&\notag\quad- C \chi'(\lambda\rho) |[\varphi] d\rho| - C \lambda^{-1}.
\end{align*} 
By assumption, the difference $A_g - A_{\tilde{g}}$ is strictly $2$-positive on $\partial M$. Therefore, the restriction of $S$ to the tangent space of $\partial M$ is strictly $2$-positive (with respect to the metric $g$). Hence, we can find a tensor $\tilde{S}$ such that $\tilde{S}-S$ is a large (but fixed) multiple of $d\rho \otimes d\rho$, and $\tilde{S}$ is strictly $2$-positive (with respect to $g$) at each point on $\partial M$. Let us fix a small number $a>0$ such that $\tilde{S}-2ag$ is $2$-positive (with respect to $g$) in a small neighborhood of $\partial M$. Then $\tilde{S}-ag$ is $2$-positive with respect to $\hat{g}_\lambda$, if $\lambda$ is sufficiently large. Since $\varphi$ is an isotropic $2$-vector with respect to $\hat{g}_\lambda$, we know that $[\varphi] d\rho$ is an isotropic vector with respect to $\hat{g}_\lambda$. Since $\tilde{S}-ag$ is $2$-positive with respect to $\hat{g}_\lambda$, we obtain 
\[\tilde{S}([\varphi] d\rho,\overline{[\varphi] d\rho}) \geq a \, |[\varphi] d\rho|^2\] 
in a neighborhood of $\partial M$. On the other hand, since $[\varphi]$ is anti-symmetric, we have $([\varphi] d\rho)^i D_i \rho = \varphi^{ij} D_i \rho D_j \rho = 0$. In other words, the vector $[\varphi] d\rho$ is annihilated by the $1$-form $d\rho$. Since $\tilde{S}-S$ is a multiple of $d\rho \otimes d\rho$, it follows that 
\[S([\varphi] d\rho,\overline{[\varphi] d\rho}) = \tilde{S}([\varphi] d\rho,\overline{[\varphi] d\rho}) \geq a \, |[\varphi] d\rho|^2\] 
in a neighborhood of $\partial M$. Then we proceed as in Proposition 3.1, we can find a real number $s_0\in (0,1)$ so that the estimate 
\begin{align*}
    \inf_{\rho\geq s_0\lambda^{-1}}(Rm_{\hat{g}_{\lambda}}(\varphi,\bar{\varphi})-Rm_{g}(\varphi,\bar{\varphi}))\geq -\epsilon
\end{align*}
holds if $\lambda > 0$ is sufficiently large.

Next, proceed as in Proposition 3.1, we have
\begin{align*} 
&\varphi^{ij} \bar{\varphi}^{kl} (R_{\hat{g}_\lambda})_{ijkl} - \varphi^{ij} \bar{\varphi}^{kl} R_{ijkl} \\ 
&\geq 2a \lambda (-\chi''(\lambda\rho)) \, |[\varphi] d\rho|^2 \\ 
&\quad+ \varphi^{ij} \bar{\varphi}^{kl} \chi'(\lambda\rho) (1 - \frac{1}{2} \chi'(\lambda\rho) |D\rho|_{\hat{g}_\lambda}^2) S_{ik} S_{jl} \\ 
&\quad + \varphi^{ij} \bar{\varphi}^{kl} \chi'(\lambda\rho) T_{ik} S_{jl} - C \chi'(\lambda\rho) |[\varphi] d\rho| - C \lambda^{-1} 
\end{align*} 
in the region $\{e^{-\lambda^2} \leq \rho \leq s_0\lambda^{-1}\}$, where the (0,2)-tensor $T$ is defined in the same way as in Proposition 3.1. In this case, $T$ is weakly 2-positive as $A_{\tilde{g}}$ does. Since $\tilde{S}$ is $2$-positive with respect to the metric $g$, by Lemma 2.3 we know that the Kulkarni-Nomizu products $\tilde{S} \owedge \tilde{S}$ and $T \owedge \tilde{S}$ are strictly PIC and weakly PIC respectively with respect to the metric $g$. Let us fix a small number $b>0$ so that $\tilde{S} \owedge \tilde{S} - 5b \, g \owedge g$ is PIC with respect to the metric $g$. Hence, if $\lambda$ is sufficiently large, then $\ \tilde{S} \owedge \tilde{S} - 4b \, g \owedge g$ and $\ T \owedge \tilde{S} + b \, g \owedge g$ are PIC with respect to the metric $\hat{g}_\lambda$. Since $\varphi$ is an isotropic $2$-vector with respect to $\hat{g}_\lambda$, it follows that 
\[\varphi^{ij} \bar{\varphi}^{kl} \tilde{S}_{ik} \tilde{S}_{jl} \geq 4b |\varphi|^2 = 4b\] 
and 
\[\varphi^{ij} \bar{\varphi}^{kl} T_{ik} \tilde{S}_{jl} \geq -b |\varphi|^2 = -b.\]

Since $\tilde{S}-S$ is a multiple of $d\rho \otimes d\rho$, we obtain 
\[\varphi^{ij} \bar{\varphi}^{kl} S_{ik} S_{jl} \geq \varphi^{ij} \bar{\varphi}^{kl} \tilde{S}_{ik} \tilde{S}_{jl}  - C \, |[\varphi] d\rho| \geq 4b - C \, |[\varphi] d\rho|\] 
and
\[\varphi^{ij} \bar{\varphi}^{kl} T_{ik} S_{jl} \geq \varphi^{ij} \bar{\varphi}^{kl} T_{ik} \tilde{S}_{jl}  - C \, |[\varphi] d\rho| \geq -b - C \, |[\varphi] d\rho|\] 
in a neighborhood of $\partial M$. To summarize, we find 
\begin{align*} 
&\varphi^{ij} \bar{\varphi}^{kl} (R_{\hat{g}_\lambda})_{ijkl} - \varphi^{ij} \bar{\varphi}^{kl} R_{ijkl} \\ 
&\geq 2a \lambda (-\chi''(\lambda\rho)) \, |[\varphi] d\rho|^2 + b \chi'(\lambda\rho) (3 - 2 \chi'(\lambda\rho) |D\rho|_{\hat{g}_\lambda}^2) \\ 
&\quad- N \chi'(\lambda\rho) |[\varphi] d\rho| - N \lambda^{-1}
\end{align*} 
in the region $\{e^{-\lambda^2} \leq \rho \leq s_0\lambda^{-1}\}$. The assertion then follows as in Proposition 3.1.

\end{proof}

\bigskip
\bigskip
\bigskip
Next, we estimate the curvature operator of $\hat{g}_{\lambda}$ in the region $\{\rho<e^{-\lambda^2}\}$.
\bigskip
\begin{proposition}
 Suppose that $(M,g)$ has a convex boundary and that $(M, \tilde{g})$ has a  weakly convex boundary such that $A_g > A_{\tilde{g}} \geq 0$. Let $\epsilon$ be an arbitrary positive real number. If $\lambda>0$ is sufficientlt large, then 
\begin{align*}
    Rm_{\hat{g}_{\lambda}}(x)(\varphi,\bar{\varphi})-Rm_{\tilde{g}}(x)(\varphi,\bar{\varphi})\geq -\epsilon|\varphi|^2_{\tilde{g}}
\end{align*}
for any (2,0)-tensor $\varphi\in\bigwedge^2T_xM^{\mathbb{C}}$ and any $x\in M$ in the region $\{\rho(x)< e^{-\lambda^2}\}$. 
\end{proposition}
\bigskip

\begin{proof}
In the region $\{\rho<e^{-\lambda^2}\}$, we have $\hat{g}_{\lambda}=\tilde{g}+\tilde{h}_{\lambda}$, where $\tilde{h}_{\lambda}$ is defined by 
\begin{align*}
    \tilde{h}_{\lambda} = -\lambda\rho^2\beta(\lambda^{-2}\log\rho)S.
\end{align*}

Let $\{e_i\}$ be a geodesic normal frame around $x$. And without loss of generality let  $\varphi\in\bigwedge^2T_xM^{\mathbb{C}}$ be a normalized (2,0)-tensor, we can write $\varphi=\sum_{i<j}\varphi^{ij}e_i\wedge e_j$. Then we have

\begin{align}
    (\tilde{D}_{e_j}\tilde{h}_{\lambda})(e_k,e_l)&=-\big[2\lambda\rho\beta(\lambda^{-2}\log\rho)+\lambda^{-1}\rho\beta'(\lambda^{-2}\log\rho)\big]D_j\rho\cdot S_{kl}\\
    &\quad - \lambda\rho^2\beta(\lambda^{-2}\log\rho)D_jS_{kl}\notag\\
    &= O(\lambda^{-1}) \notag
\end{align}
and 
\begin{align} 
    (\tilde{D}^2_{e_i,e_j}\tilde{h}_{\lambda})(e_k,e_l)&= -\big[2\lambda\beta(\lambda^{-2}\log\rho)+3\lambda^{-1}\beta'(\lambda^{-2}\log\rho)+\lambda^{-3}\beta''(\lambda^{-2}\log\rho)\big]D_i\rho D_j\rho\cdot S_{kl} \\
    &\quad -\big[2\lambda\rho\beta(\lambda^{-2}\log\rho) + \lambda^{-1}\rho\beta'(\lambda^{-2}\log\rho) \big]D_iD_j\rho\cdot S_{kl}    \notag\\
    &\quad -\big[2\lambda\rho\beta(\lambda^{-2}\log\rho) + 2\lambda^{-1}\rho\beta'(\lambda^{-2}\log\rho)\big]D_j\rho\cdot D_iS_{kl} \notag\\
    &\quad  
    - \big[2\lambda\rho\beta(\lambda^{-2}\log\rho) + 2\lambda^{-1}\rho\beta'(\lambda^{-2}\log\rho)\big]D_i\rho\cdot D_jS_{kl}\notag\\
    &\quad -\lambda\rho^2\beta(\lambda^{-2}\log\rho)D_iD_jS_{kl}\notag\\
    &= -2\lambda\beta(\lambda^{-2}\log\rho)D_i\rho D_j\rho\cdot S_{kl} + O(\lambda^{-1}). \notag 
\end{align}

Using Lemma 2.1, we obtain 
\begin{align}
    Rm_{\hat{g}_{\lambda}}(\varphi,\bar{\varphi})-Rm_{g}(\varphi,\bar{\varphi}) \geq\ &  2\lambda\beta(\lambda^{-2}\log\rho)S([\varphi]d\rho,\overline{[\varphi]d\rho}) - L\lambda^{-1}.
\end{align}
Here, $L$ is a positive constant independent of $\lambda$.
\bigskip

Recall that $S(X,X)>0$ for all $X\in T(\partial M)^{\mathbb{C}}$. By continuity, we have

\begin{align*}
    S([\varphi]d\rho,\overline{[\varphi]d\rho})\geq 0
\end{align*}

in a neighborhood of $\partial M$. Hence, if $\lambda>0$ is sufficiently large, then we have
\begin{align*}
    \inf_{\rho<e^{-\lambda^2}}(Rm_{\hat{g}_{\lambda}}(\varphi,\bar{\varphi})-Rm_{\tilde{g}}(\varphi,\bar{\varphi}))\geq -\epsilon.
\end{align*}
From this, the assertion follows.

\end{proof}
\bigskip

Following the same line of the above Proposition and Corollary 3.2, we also have

\begin{corollary}
 Suppose that $(M,g)$ has a two-convex boundary and that $(M, \tilde{g})$ has a  weakly two-convex boundary such that 
	$$ A_g(X,X)+A_g(Y,Y) > A_{\tilde{g}}(X,X) + A_{\tilde{g}}(Y,Y) \geq 0 $$
	for all orthonormal $X,Y\in T(\partial M)$. Let $\epsilon$ be an arbitrary positive real number. If $\lambda>0$ is sufficientlt large, then 
\begin{align*}
    Rm_{\hat{g}_{\lambda}}(x)(\varphi,\bar{\varphi})-Rm_{\tilde{g}}(x)(\varphi,\bar{\varphi})\geq -\epsilon|\varphi|^2_{\tilde{g}}
\end{align*}
for any isotropic 2-vector $\varphi=z\wedge w\in\bigwedge^2T_xM^{\mathbb{C}}$ satisfying $\hat{g}_{\lambda}(z,z)=\hat{g}_{\lambda}(w,w)=\hat{g}_{\lambda}(z,w)=0$ and any $x\in M$ in the region $\{\rho(x)< e^{-\lambda^2}\}$. 
\end{corollary}

\bigskip

We now sum up the results in this section:
\begin{corollary}
Let $\epsilon>0$ be an arbitrary positive real number. 
\begin{itemize}
\item[(i)]  Suppose that $(M,g)$ has a convex boundary and that $(M, \tilde{g})$ has a  weakly convex boundary such that $A_g > A_{\tilde{g}} \geq 0$. If $\lambda>0$ is sufficiently large, then we have the point-wise inequality	
	\begin{align*}
    Rm_{\hat{g}_{\lambda}}(x)(\varphi,\bar{\varphi}) \geq \min\{Rm_g(x)(\varphi,\bar{\varphi}), Rm_{\tilde{g}}(x)(\varphi,\bar{\varphi})\} - \epsilon\max\{|\varphi|^2_g,|\varphi|^2_{\tilde{g}}\}
	\end{align*}
	for all $\varphi\in \bigwedge^2T_xM^{\mathbb{C}}$ and each point $x\in M$.
\item[(ii)]  Suppose that $(M,g)$ has a two-convex boundary and that $(M, \tilde{g})$ has a  weakly two-convex boundary such that 
	$$ A_g(X,X)+A_g(Y,Y) > A_{\tilde{g}}(X,X) + A_{\tilde{g}}(Y,Y) \geq 0 $$
	for all orthonormal $X,Y\in T(\partial M)$. If $\lambda>0$ is sufficiently large, then we have the point-wise inequality
	\begin{align*}
    Rm_{\hat{g}_{\lambda}}(x)(\varphi,\bar{\varphi}) \geq \min\{Rm_g(x)(\varphi,\bar{\varphi}), Rm_{\tilde{g}}(x)(\varphi,\bar{\varphi})\} - \epsilon\max\{|\varphi|^2_g,|\varphi|^2_{\tilde{g}}\}
	\end{align*}
	for all isotropic 2-vector $\varphi=z\wedge w\in\bigwedge^2T_xM^{\mathbb{C}}$ satisfying $\hat{g}_{\lambda}(z,z)=\hat{g}_{\lambda}(w,w)=\hat{g}_{\lambda}(z,w)=0$ and each point $x\in M$.
\item[(iii)]  Suppose that $(M,g)$ has a mean-convex boundary and that $(M, \tilde{g})$ has a  weakly mean-convex boundary such that $H_g \geq H_{\tilde{g}} >0$. If $\lambda>0$ is sufficiently large, then we have the point-wise inequality	
	\begin{align*}
    R_{\hat{g}_{\lambda}}(x) \geq \min\{R_g(x), R_{\tilde{g}}(x)\} - \epsilon
	\end{align*}
	for each point $x\in M$.
\end{itemize}
\end{corollary}

Assertion (iii) in above corollary was given by Theorem 5 in \cite{BMN}. 

\bigskip
\bigskip

\section{Proof of Main Theorem 1}

Assertion (i) in Main Theorem 1 follows from the definition of $\hat{g}_{\lambda}$ in (2). Assertion (iv) follows from Corollary 3.5 for sufficiently large $\lambda>0$.
\begin{proof}[Proof of (ii):]\ 

 Suppose that $(M,g)$ has a convex boundary and that $(M, \tilde{g})$ has a  weakly convex boundary such that $A_g > A_{\tilde{g}} \geq 0$. Suppose also that $(M,g)$ and $(M, \tilde{g})$ are PIC1. Let $\varphi=z\wedge w\in \bigwedge^2T_xM^{\mathbb{C}}$  satisfies the condition $\hat{g}_{\lambda}(z,z)\hat{g}_{\lambda}(w,w) - \hat{g}_{\lambda}(z,w)^2 =0$. This condition is equivalent to the condition $\sum_{i,k}\varphi^{ik}\varphi^{ki}=0$ with respect to the metric $\hat{g}_{\lambda}$. We wish to prove that $Rm_{\hat{g}_{\lambda}}(x)(\varphi,\bar{\varphi})> 0$. We divide the proof into two cases: (1) $x$ is in the region $\{\rho\geq e^{-\lambda^2}\}$; (2) $x$ is in the region $\{\rho < e^{-\lambda^2}\}$.

For the first case, we have $\hat{g}_{\lambda} = g+h_{\lambda}$ where $h_{\lambda}=\lambda^{-1}\chi(\lambda\rho)S$. We fix an orthonormal frame $\{e_i\}$ with respect to $g$ such that this frame diagonalizes $h_{\lambda}$. i.e. $h_{ij}=\delta_{ij}\mu_j$, where $\mu_j$ are eigenvalues of $h$. We then evolve the frame $\{e_i\}$ in $T_xM$ by
\begin{align}
    \begin{cases}
    \frac{d}{ds}E_i(s) &=\ -\frac{1}{2}h_{\lambda}\circ E_i(s)\\
    E_i(0)&=\ e_i.
    \end{cases}
\end{align}
Then we see that $\{E_i(s)\}$ remains orthonormal with respect to $g_s=g+sh_{\lambda}$. In particular, we have $\{E_i(1)\}$ being orthonormal with respect to $\hat{g}_{\lambda}=g+h_{\lambda}$.

Now, we set $\varphi_s=\sum_{ij}\varphi^{ij}E_i(s)\wedge E_j(s)$ so that $\varphi_1 = \varphi$ and we define $\psi = \varphi_0 = \sum_{ij}\varphi^{ij}E_i(0)\wedge E_j(0)$. Hence we see that the condition $ \sum_{i,k}\varphi^{ik}\varphi^{ki} =0$ implies the condition $ \sum_{i,k}\psi^{ik}\psi^{ki} =0$ with respect to the metric $g$. Since $(M,g)$ is PIC1, it follows that $Rm_g(x)(\psi,\bar{\psi})> 0$. In view of Corollary 3.5, it suffices to show that  $Rm_g(x)(\varphi,\bar{\varphi})> 0$.

\bigskip
By writing $E_i(s)=A_i^k(s)e_k$, we observe that
\begin{align*}
    \frac{d}{ds}E_i(s) = \frac{d}{ds}A_i^k(s)e_k=-\frac{1}{2}h_l^kA_i^l(s)e_k=-\frac{1}{2}\sum_k\mu_kA^k_i(s)e_k
\end{align*}

for any $s\in[0,1]$. Then
\begin{align*}
    \frac{d}{ds}\varphi_s=2\varphi^{ij}E'_i(s)\wedge E_j(s) = -\mu_k\varphi^{ij}A_i^k(s)e_k\wedge E_j(s).
\end{align*}

So we have the estimate
\begin{align*}
    &\Big|\frac{d}{ds}Rm_g(\varphi_s,\bar{\varphi_s})\Big|\\
    &=\sum_k|\mu_k|\Big|Rm_g(\varphi^{ij}A_i^k(s)e_k\wedge E_j(s),\bar{\varphi_s}) + Rm_g(\varphi_s, \overline{\varphi^{ij}A_i^k(s)e_k\wedge E_j(s)}) \Big|\\
    &\leq\frac{C}{\lambda}.
\end{align*}

Here $C$ is a positive constant depending only on $(M,g)$. This implies that

\begin{align*}
    \Big|R_g(\varphi_1,\bar{\varphi_1}) - R_g(\varphi_0,\bar{\varphi_0})\Big|
    &\leq \int_0^1\frac{C}{\lambda}d\tau\leq \frac{C}{\lambda}. \notag
\end{align*}
Hence,
\begin{align}
    Rm_g(x)(\varphi,\bar{\varphi})\geq \frac{1}{2}Rm_g(x)(\psi,\bar{\psi})> 0
\end{align}
if $\lambda>0$ is sufficiently large. 

For case (2), we have $\hat{g}_{\lambda}=\tilde{g}- \lambda\rho^2\beta(\lambda^{-2}\log\rho)S$. Thus $h_{\lambda} = -\lambda\rho^2\beta(\lambda^{-2}\log\rho)S$. Following the same argument in case (1) and the fact that 
\begin{align*}
    |h_{\lambda}|\leq C(g)e^{-\lambda^2}\leq \frac{C(g)}{\lambda}
\end{align*}
in the region $\rho<e^{-\lambda^2}$ for sufficiently large $\lambda>0$,  we also have
$$Rm_{\tilde{g}}(x)(\varphi,\bar{\varphi})\geq \frac{1}{2}Rm_{\tilde{g}}(x)(\psi,\bar{\psi})$$
for sufficiently large $\lambda$. Since $(M, \tilde{g})$ is also PIC1, this gives $Rm_{\tilde{g}}(x)(\varphi,\bar{\varphi})>0$.

Combining the two cases together,  Corollary 3.5 implies that
\begin{align*}
    Rm_{\hat{g}_{\lambda}}(x)(\varphi,\bar{\varphi}) >0
\end{align*}
for sufficiently large $\lambda>0$. From this, we conclude that $(M,\hat{g}_{\lambda})$ is also PIC1 for sufficiently large $\lambda$. The other assertions in statement (ii) of the Main Theorem 1 can be proved similarly. 
	
\end{proof}

\bigskip
\begin{proof}[Proof of (iii):]\

 Suppose that $(M,g)$ has a two-convex boundary and that $(M, \tilde{g})$ has a  weakly two-convex boundary such that 
	$$ A_g(X,X)+A_g(Y,Y) > A_{\tilde{g}}(X,X) + A_{\tilde{g}}(Y,Y) \geq 0 $$
	for all orthonormal $X,Y\in T(\partial M)$. Suppose also that $(M,g)$ and $(M, \tilde{g})$  are  PIC. Let $\varphi=z\wedge w\in \bigwedge^2T_xM^{\mathbb{C}}$ be an isotropic 2-vector satisfying $\hat{g}_{\lambda}(z,z) = \hat{g}_{\lambda}(w,w) = \hat{g}_{\lambda}(z,w) =0$. This condition is equivalent to the condition $\sum_{k}\varphi^{ik}\varphi^{kj}=0$ with respect to the metric $\hat{g}_{\lambda}$. Using the same argument in the previous proof, it follows from Lemma 2.1 and Corollary 3.5 that we have $R_{\hat{g}_{\lambda}}(\varphi, \bar{\varphi})> 0$ for sufficiently large $\lambda>0$. Thus $(M,\hat{g}_{\lambda})$ is also PIC. The assertion then follows.

\end{proof}

\bigskip

\section{Proof of Main Theorem 2}
\bigskip

Since the boundary $\partial M$ is an embedded hypersurface in $M$, we can use the Tubular Neighborhood Theorem to find an open neighborhood $U$ of $\partial M$ and a Riemannian metric $\tilde{g}$ on $U$, such that $g-\tilde{g}=0$ at each point of $\partial M$, and that the boundary $\partial M$ is totally geodesic with respect to $\tilde{g}$. It follows from the Lemma below that we can choose $\tilde{g}$ such that it satisfies the same curvature conditions as $g$ does. Having fixed the neighborhood $U$, we define $\hat{g}_{\lambda}$ as in (2). The metric $\hat{g}_{\lambda}$ is well-defined for sufficiently large $\lambda$. Finally, Main Theorem 2 follows from applying Main Theorem 1 with choosing the metric $\tilde{g}$ to be the one constructed in the following Lemma. 
\bigskip

\begin{lemma}
Suppose that $(M,g)$ is a smooth, compact Riemannian manifold with boundary $\partial M$. Then there is a neighborhood $U$ of $\partial M$ and a Riemannian metric $\tilde{g}$ on $U$ such that
\begin{itemize}
    \item [(i)] $\tilde{g}-g=0$ on $\partial M$.
    \item[(ii)] $\partial M$ is totally geodesic with respect to $\tilde{g}$.
    \item[(iii)] If $(M,g)$ has a convex boundary, then
	\begin{itemize}
	\item $(M,g)$ has positive curvature operator $\implies$ $(U,\tilde{g})$ has positive curvature operator;
    \item $(M,g)$ is PIC1 $\implies$ $(U, \tilde{g})$ is PIC1;
    \item $(M,g)$ is PIC2 $\implies$ $(U,\tilde{g})$ is  PIC2.
	\end{itemize}	
\item[(iv)] If $(M,g)$ has a 2-convex boundary, then
	\begin{itemize}
	\item $(M,g)$ is PIC $\implies$ $(U,\tilde{g})$ is PIC.
	\end{itemize}
\item[(v)] $(U,\tilde{g})$ has positive scalar curvature.
\end{itemize}
\end{lemma}
\bigskip

\begin{proof}
By the tubular neighbourhood theorem, there is a neighbourhood $U$ of $\partial M$ that is diffeomorphic to an open set of the normal bundle of $\partial M$. Specifically, there is a diffeomorphism $\Phi:U\to V$ where $V=\{(x,s)\in\partial M\times\mathbb{R}_+: |s|<\delta\}$ and $\Phi(\partial M)=\partial M\times\{0\}$. Let $\theta>0$ to be a large constant specified later, we define $V_{\theta}\subset V$ by $V_{\theta}=\{(x,s)\in V: |s|<\delta\theta^{-3}\}$. Define on $V_{\theta}$ a product metric $ds^2+(\cos^2(\theta s)) g_{\partial M}$.   Now we define a metric $\tilde{g}$ on $U_{\theta}=\Phi^{-1}(V_{\theta})$ by
\begin{align}
    \tilde{g}=\Phi^*(ds^2+(\cos^2(\theta s)) g_{\partial M}).
\end{align}
It is easy to see that condition (i) is satisfied. Next, let $\{E_1,..,E_n\}$ be a local frame of $TU_{\theta}$ such that $\{d\Phi(E_a)\}_{a=1,..,n-1}$ is an O.N. frame with respect to $g_{\partial M}$ and $d\Phi(E_n)=\frac{\partial}{\partial s}$. From now on, the subscript $a,b,c,..$ varies over $1$ to $n-1$. Using the Koszul formula, we compute the connection terms of $\tilde{g}$:
\begin{align}
    &\tilde{g}(\tilde{D}_{E_a}E_b, E_c)= \cos^2(\theta s)g_{\partial M}(D_{d\Phi(E_a)}d\Phi(E_b), d\Phi(E_c)),\\
    &\tilde{g}(\tilde{D}_{E_a}E_b, E_n)= \frac{1}{2}\theta\sin(2\theta s)\delta_{ab},\notag\\
    &\tilde{g}(\tilde{D}_{E_a}E_n, E_b)= -\frac{1}{2}\theta\sin(2\theta s)\delta_{ab},\notag\\
    &\tilde{g}(\tilde{D}_{E_a}E_n, E_n)=0\notag,\\
    &\tilde{g}(\tilde{D}_{E_n}E_a, E_b)=-\frac{1}{2}\theta\sin(2\theta s)\delta_{ab},\notag\\
    &\tilde{g}(\tilde{D}_{E_n}E_a, E_n)=0,\notag\\
    &\tilde{D}_{E_n}E_n=0.\notag
\end{align}
Identifying $\Phi$ as the identity map, we then have
\begin{align}
    &\tilde{D}_{E_a}E_b = D_{E_a}E_b+\frac{1}{2}\theta\sin(2\theta s)\delta_{ab}E_n,\\
    &\tilde{D}_{E_a}E_n = -\theta\tan(\theta s)E_a,\notag\\
    &\tilde{D}_{E_n}E_a = -\theta\tan(\theta s)E_a,\notag\\
    &\tilde{D}_{E_n}E_n = 0.\notag
\end{align}
From this, we calculate the components of the Riemann curvature tensor of $\tilde{g}$:
\begin{align}
    &\tilde{R}_{abcd}=\cos^2(\theta s)(R_{g_{\partial M}})_{abcd}-\theta^2\sin^2(\theta s)\cos^2(\theta s)(\delta_{ac}\delta_{bd}-\delta_{ad}\delta_{bc}),\\
    &\tilde{R}_{nbnd} = \theta^2\cos^2(\theta s)\delta_{bd},\notag\\
    &\tilde{R}_{abnd} = 0.\notag
\end{align}
Now, it is easy to see from (17) that
\begin{align*}
    \tilde{D}_{E_a}E_b = D_{E_a}E_b+\frac{1}{2}\theta\sin(0)\delta_{ab}E_n = D_{E_a}E_b
\end{align*}
on $\partial M$. Hence condition (ii) is satisfied by $\tilde{g}$. 
\bigskip

Next, we fix a point $p\in U_{\theta}$. Note that the diffeomorphism $\Phi$ maps $p$ to a point $(\bar{p},s)\in\partial M\times \mathbb{R}_+$. Let $\varphi\in\bigwedge^2T_pU_{\theta}$ be a unit two-vector with respect to $\tilde{g}$. Since the vector fields $\tilde{E}_a = \frac{1}{\cos(\theta s)}E_a$ and $\tilde{E_n}=E_n$ form an orthonormal frame in $U_{\theta}$ with respect to $\tilde{g}$, we can write 
\begin{align*}
\varphi &=\sum_{a<b}\varphi^{ab}\tilde{E}_a\wedge \tilde{E}_b + \sum_{a<n}\varphi^{an}\tilde{E}_a\wedge \tilde{E}_n \\
	&= 	\sum_{a<b}\frac{1}{\cos^2(\theta s)}\varphi^{ab}E_a\wedge E_b + \sum_{a<n}\frac{1}{\cos(\theta s)}\varphi^{an}E_a\wedge E_n.
\end{align*}
 And we also denote $\varphi^T = \sum_{a<b}\varphi^{ab}\tilde{E}_a\wedge \tilde{E}_b$ and $\varphi^N = \sum_{a<n}\varphi^{an}\tilde{E}_a\wedge \tilde{E}_n$. Then by (18) we have
 \begin{align}
    R_{\tilde{g}}(p)(\varphi,\bar{\varphi}) &= \sum_{a<b}\sum_{c<d}\frac{1}{\cos^4(\theta s)}\varphi^{ab}\bar{\varphi}^{cd}\tilde{R}_{abcd}(p) + \sum_{a<n}\sum_{c<n}\frac{1}{\cos^2(\theta s)}\varphi^{an}\bar{\varphi}^{cn}\tilde{R}_{ancn}(p)\\
    &\notag= \sum_{a<b}\sum_{c<d}\frac{1}{\cos^2(\theta s)}\varphi^{ab}\bar{\varphi}^{cd}(R_{g_{\partial M}})_{abcd}(\bar{p})\\
    &\notag\quad  - \theta^2\tan^2(\theta s)\sum_{a<b}\sum_{c<d}\varphi^{ab}\bar{\varphi}^{cd}(\delta_{ac}\delta_{bd}-\delta_{ad}\delta_{bc}) +\theta^2\sum_{a<n}\sum_{c<n}\varphi^{an}\bar{\varphi}^{cn}\delta_{ac}\\
    &\notag=cos^2(\theta s)R_{g_{\partial M}}(\bar{p})(d\Phi(\varphi^T),d\Phi(\bar{\varphi}^T)) - \theta^2\tan^2(\theta s)|\varphi^T|^2_{\tilde{g}}+\theta^2|\varphi^N|^2_{\tilde{g}}.
\end{align}
Here $d\Phi(\varphi^T) = \sum_{a<b}\frac{1}{\cos^2(\theta s)}\varphi^{ab}d\Phi(E_a)\wedge d\Phi(E_b)$. By the Gauss equation, we have
\begin{align}
    &R_{g_{\partial M}}(\bar{p})(d\Phi(\varphi^T),d\Phi(\bar{\varphi}^T))\\
    &\notag= R_g(\bar{p})(d\Phi(\varphi^T),d\Phi(\bar{\varphi}^T)) + \sum_{a<b}\sum_{c<d}\frac{1}{\cos^4(\theta s)}\varphi^{ab}\bar{\varphi}^{cd}(A_{bd}A_{ac}-A_{ad}A_{bc}).
\end{align}
Here $A=A_g$ is the second fundamental form of the boundary with respect to $g$. Now, associating to $\varphi$ we define a two-vector $\psi\in\bigwedge^2T_{\bar{p}}M$ by
\begin{align}
\psi := 	\sum_{a<b}\frac{1}{\cos^2(\theta s)}\varphi^{ab}d\Phi(E_a)\wedge d\Phi(E_b) + \sum_{a<n}\frac{1}{\cos^2(\theta s)}\varphi^{an}d\Phi(E_a)\wedge\nu,
\end{align}
where $\nu$ is the unit normal vector field on $\partial M$ with respect to $g$. Hence $\{d\Phi(E_a),\nu\}$ forms an orthonormal frame around $\bar{p}\in\partial M$. Thus,
\begin{align}
	&R_g(\bar{p})(d\Phi(\varphi^T),d\Phi(\bar{\varphi}^T))\\
	&\notag= R_g(\bar{p})\left(\psi - \sum_{a<n}\frac{1}{\cos^2(\theta s)}\varphi^{an}d\Phi(E_a)\wedge\nu
,\quad\overline{\psi - \sum_{a<n}\frac{1}{\cos^2(\theta s)}\varphi^{an}d\Phi(E_a)\wedge\nu
}\right)\\
	&\notag > R_g(\bar{p})(\psi, \bar{\psi}) - C_1(g)|\varphi^N|_{\tilde{g}}.
\end{align}
Here $C(g)$ is a uniform constant depending only on $(M,g)$. Upon combining (19), (20) and (22) we obtain the following estimate
\begin{align}
	R_{\tilde{g}}(p)(\varphi,\bar{\varphi})&> \frac{1}{2}R_g(\bar{p})(\psi, \bar{\psi}) + \sum_{a<b}\sum_{c<d}\varphi^{ab}\bar{\varphi}^{cd}(A_{bd}A_{ac}-A_{ad}A_{bc})\\
	&\notag\quad - \theta^2\tan^2(\theta s)|\varphi^T|^2_{\tilde{g}}+\theta^2|\varphi^N|^2_{\tilde{g}} - C_1(g)|\varphi^N|_{\tilde{g}}.
\end{align}
Next, by noting that $\theta^3s<\delta \implies \theta\tan(\theta s)<B\theta^2s<B\delta\theta^{-1}$ for some positive constant $B$, we have
\begin{align}
	R_{\tilde{g}}(p)(\varphi,\bar{\varphi})&> \frac{1}{2}R_g(\bar{p})(\psi, \bar{\psi}) + \sum_{a<b}\sum_{c<d}\varphi^{ab}\bar{\varphi}^{cd}(A_{bd}A_{ac}-A_{ad}A_{bc})- \frac{B^2\delta^2}{\theta^2}|\varphi^T|^2_{\tilde{g}}\\
	&\notag\quad+\theta^2|\varphi^N|^2_{\tilde{g}} - C_1(g)|\varphi^N|_{\tilde{g}}.
\end{align}
Now, we are ready to prove statements (iii) and (iv) of the Lemma.

\bigskip

\noindent\textsl{Proof of (iii):}

Suppose that $(M,g)$ has convex boundary, then the second term in the RHS of (24) is positive, thus
\begin{align}
	R_{\tilde{g}}(p)(\varphi,\bar{\varphi})&> \frac{1}{2}R_g(\bar{p})(\psi, \bar{\psi}) + (\inf_{\partial M}\mu_1)^2|\varphi^T|^2_{\tilde{g}} - \frac{B^2\delta^2}{\theta^2}|\varphi^T|^2_{\tilde{g}}+\theta^2|\varphi^N|^2_{\tilde{g}} - C_1(g)|\varphi^N|_{\tilde{g}},
\end{align}
where $\mu_1$ is the smallest eigenvalue of $A_g$. If $(M,g)$ is PIC1, then we let $\varphi=z\wedge w\in\bigwedge^2T_pU^{\mathbb{C}}$ be a unit two-vector such that $\tilde{g}(z,z)\tilde{g}(w,w)-\tilde{g}(z,w)^2 = 0$. This condition is equivalent to the condition $\sum_{i<j}\varphi^{ij}\varphi^{ij}=0$ with respect to the metric $\tilde{g}$. By the definition of $\psi$, this implies that $\sum_{i<j}\psi^{ij}\psi^{ij} = 0$ with respect to the metric $g$. Since $(M,g)$ is PIC1, $R_g(\bar{p})(\psi, \bar{\psi})>0$, we obtain
\begin{align*}
	R_{\tilde{g}}(p)(\varphi,\bar{\varphi})&> (\inf_{\partial M}\mu_1)^2|\varphi|^2_{\tilde{g}} - \frac{B^2\delta^2}{\theta^2}|\varphi^T|^2_{\tilde{g}}+(\theta^2-(\inf_{\partial M}\mu_1)^2)|\varphi^N|^2_{\tilde{g}} - C_1(g)|\varphi^N|_{\tilde{g}}\\
	&> (\inf_{\partial M}\mu_1)^2|\varphi|^2_{\tilde{g}} - \frac{B^2\delta^2}{\theta^2}|\varphi^T|^2_{\tilde{g}} - \frac{C_1(g)^2}{\theta^2}\\
	&> \frac{1}{2}(\inf_{\partial M}\mu_1)^2|\varphi|^2_{\tilde{g}},
\end{align*}
provided that $\theta$ is sufficiently large. Hence $(U,\tilde{g})$ is also PIC1. The other two assertions under statement (iii) can be proved similarly.

\bigskip
\noindent\textsl{Proof of (iv):}

Suppose now that $(M,g)$ has a two-convex boundary and is PIC. Let $\varphi=z\wedge w\in\bigwedge^2T_pU^{\mathbb{C}}$ be a unit two-vector such that $\tilde{g}(z,z) = \tilde{g}(w,w)=\tilde{g}(z,w) = 0$. By a similar argument as in the proof of (iii), the definition of $\psi$ and the fact that $(M,g)$ is PIC imply the positivity of $R_g(\bar{p})(\psi, \bar{\psi})$. We will show that the second term in $RHS$ of (24) is positive modulus an error term involving the normal components. At the point $\bar{p}\in\partial M$, the restriction of $A_g$ to the tangent space $T_{\bar{p}}(\partial M)$ is 2-positive. We define a (0,2)-tensor $\tilde{A}_g$ by $\tilde{A}_g = A_g + \nu^*\otimes \nu^*$, thus $\tilde{A}_g$ is strictly 2-positive on $T_{\bar{p}}M$. Consequently, the Kulkarni-Nomizu product $\tilde{A}_g\owedge\tilde{A}_g$ is strictly PIC by Lemma 2.3. In fact, the above procedure can be done uniformly at each point on the boundary. On the other hand, if we extend the definition of $\varphi$ so that $\varphi^{ji}=-\varphi^{ij}, i<j$, the condition $\tilde{g}(z,z) = \tilde{g}(w,w)=\tilde{g}(z,w) = 0$ is equivalent to the condition $\sum_k\varphi^{ik}\varphi^{kj}=0$ with respect to the metric $\tilde{g}$. By the definition of $\psi$, this implies that $\sum_k\psi^{ik}\psi^{kj} = 0$ with respect to the metric $g$.  Therefore, we can find a uniform constant $a>0$ such that
\begin{align*}
	\tilde{A}_g\owedge\tilde{A}_g(\psi, \bar{\psi}) > a|\psi|_g^2
\end{align*}
 on $T_{\bar{p}}M$. It follows from the definition of $\tilde{A}_g$ that 
 \begin{align}
 \sum_{a<b}\sum_{c<d}\psi^{ab}\bar{\psi}^{cd}(A_{bd}A_{ac}-A_{ad}A_{bc})	 > \frac{a}{2}|\psi|_g^2 - C_2(g)|\psi^N|_g^2,
 \end{align}
where $C_2(g)$ is a uniform constant depending only on $(M,g)$. From the definition of $\psi$, we then have
 \begin{align}
 \sum_{a<b}\sum_{c<d}\varphi^{ab}\bar{\varphi}^{cd}(A_{bd}A_{ac}-A_{ad}A_{bc})	 > \frac{a}{2}|\varphi|_{\tilde{g}}^2 - C_2(g)|\varphi^N|_{\tilde{g}}^2.
 \end{align}
Combining (24), (27) and the positivity of $R_g(\bar{p})(\psi, \bar{\psi})$, we obtain 
\begin{align*}
	R_{\tilde{g}}(p)(\varphi,\bar{\varphi})&> \frac{a}{2}|\varphi|^2_{\tilde{g}} - \frac{B^2\delta^2}{\theta^2}|\varphi^T|^2_{\tilde{g}}+(\theta^2-C_2(g))|\varphi^N|^2_{\tilde{g}} - C_1(g)|\varphi^N|_{\tilde{g}}\\
	&>  \frac{a}{2}|\varphi|^2_{\tilde{g}} - \frac{B^2\delta^2}{\theta^2}|\varphi^T|^2_{\tilde{g}} - \frac{C_1(g)^2}{\theta^2}\\
	&>  \frac{a}{4}|\varphi|^2_{\tilde{g}},
\end{align*}
 provided that $\theta$ is sufficiently large. Hence $(U,\tilde{g})$ is also PIC.
 \bigskip
 
 \noindent\textsl{Proof of (v)}:
 
 We have
 \begin{align*}
 R_{\tilde{g}} &= \sum_{i,j=1}^n\tilde{R}(\tilde{E}_i, \tilde{E}_j, \tilde{E}_i, \tilde{E}_j)	\\
 &= \sum_{a,b=1}^{n-1}\frac{1}{\cos^4(\theta s)}\tilde{R}_{abab} + 2\sum_{a=1}^{n-1}\frac{1}{\cos^2(\theta s)}\tilde{R}_{anan}\\
 &= \sum_{a,b=1}^{n-1}\frac{1}{\cos^2(\theta s)}(R_{g_{\partial M}})_{abab} - (n-1)^2\theta^2\tan^2(\theta s) + 2(n-1)\theta^2\\
 &> \sum_{a,b=1}^{n-1}\frac{1}{\cos^2(\theta s)}(R_{g_{\partial M}})_{abab} - (n-1)^2\frac{B^2\delta^2}{\theta^2} + 2(n-1)\theta^2\\
 &>0,
 \end{align*}
if $\theta$ is sufficiently large.

\end{proof}

\bigskip

\providecommand{\bysame}{\leavevmode\hbox to3em{\hrulefill}\thinspace}

\providecommand{\MRhref}[2]{%
  \href{http://www.ams.org/mathscinet-getitem?mr=#1}{#2}
}
\providecommand{\href}[2]{#2}

\end{document}